\documentclass{article}
\newtheorem{lem}{Lemma}
\newtheorem{thm}{Theorem}

\newenvironment{definition}[1][Definition]{\begin{trivlist}
\item[\hskip \labelsep {\bfseries #1}]}{\end{trivlist}}

\usepackage{enumerate}
\usepackage{amssymb,amsmath}
\newcommand{\field}[1]{\mathbb{#1}}
\usepackage[numbers]{natbib}
\newenvironment{proof}[1][Proof:]{\begin{trivlist}
\item[\hskip \labelsep {\bfseries #1}]}{\end{trivlist}}
\usepackage[numbers]{natbib}
\title{Indicable Groups and Endomorphic Presentations}
\author{Mustafa G\"okhan Benli}

\begin{document}
\newcommand{\gp}[1]{ \left\langle #1 \right\rangle}
\maketitle
\begin{abstract}

In this note we look at presentations of  subgroups of finitely presented groups with infinite cyclic quotients.
We prove that if $H$ is a finitely generated normal subgroup of a finitely presented group $G$ with $G/H$ cyclic, 
then $H$ has ascending finite  endomorphic presentation. It follows that any finitely presented indicable 
group without free semigroups has the structure of a semidirect product $H \rtimes \field{Z}$ where
$H$ has finite ascending endomorphic presentation.
\end{abstract}

\section{Introduction}
It is a well known  fact that finite index subgroups of finitely presented groups are also finitely presented. 
But once one looks at subgroups of infinite index various possibilities can occur. It may be that the subgroup 
is not  finitely generated but even one can have finitely generated infinitely presented subgroups. A well known 
example is  the kernel of the map $ F_2 \times F_2 \rightarrow \field{Z}$ where each generator is mapped to 1
 (See \cite{MR760871}).

In this note we look at subgroups of finitely presented groups with infinite cyclic quotients.
The Higman embedding theorem \cite{MR0130286}, states that finitely generated subgroups of finitely 
presented groups are exactly the recursively presented groups. In the case when the subgroup has infinite 
cyclic quotient we show that it has a special recursive presentation called a  \textit{finite endomorphic presentation}
(or a finite \textit{L-presentation}). More precisely we prove the following: \par
 \bigskip
\textbf{Theorem 1} Let $G$ be a finitely presented group containing a finitely generated normal subgroup
$H$ such that $G/H$ is infinite cyclic. Then $H$ has ascending finite endomorphic presentation with two free groups
endomoprhisms.
\bigskip

Intuitively, a finite endomorphic presentation is a generalization of a finite presentation in which
the relators of the presentation are obtained by iterating a finite set of initial relators over a finite 
set of endomorphisms of the underlying free group (see next section for a precise definition). It is
yet another way of defining a group with finite data. Such presentations
first arise in the study of  self-similar groups: It was proven by Lysenok in \cite{MR819415}
that the \textit{first Grigorchuk group} $\mathcal{G}$ has the following presentation:

$$\mathcal{G}=\gp{a,b,c,d \mid a^2,b^2,c^2,d^2,bcd,\sigma^i((ad)^4),\sigma^i((adacac)^4),i\geq0}, $$
where $\sigma$ is the substitution
$$
\sigma = \left\{
\begin{array}{ccc}
   a & \mapsto & aca \\
   b & \mapsto  & d \\
   c & \mapsto & b \\
    d & \mapsto & c
  \end{array}
\right.$$

Later more examples of presentations of this kind were found for various groups including iterated
monodromy groups. (See for example \cite{MR2009317}, \cite{MR2035113}, \cite{MR2352271}  and \cite{MR1929716}).
A systematic study of such presentations was done by L. Bartholdi in \cite{MR2009317} who also 
suggested the name \textit{endomorphic presentations}. In the same paper it is also proven that any finitely generated, regular branch self-similar group has such a presentation.

Groups with finite endomorphic presentations embed nicely in finitely presented groups obtained from
the original group via 
finitely many HNN extensions \cite{MR2009317}. The first example of such an embedding was done by
Grigorchuk in \cite{MR1616436} for the group $\mathcal{G}$. Using Lysenok's presentation he showed that $\mathcal{G}$ embeds into the finitely 
presented HNN extension

$$\overline{\mathcal{G}}=\gp{\mathcal{G},t\mid t^{-1}\mathcal{G}t=\sigma(\mathcal{G})} $$

which is amenable but not elementary amenable. This showed that amenable and elementary amenable groups are separated even in the class of finitely
presented groups.

Recall that a group is termed indicable if it has a homomorphism onto the infinite cyclic group. Indicable groups
play an important role in the study of right orderable groups, amenability and bounded cohomology (See 
\cite{MR1222727}, \cite{MR2286034}, \cite{MR667000}).

A theorem of R.Bieri and R.Strebel \cite{MR1243634} (page 67) states that a finitely presented
indicable group not containing a free subgroup of rank 2, is an ascending HNN extension with 
a finitely generated base group. The group $\overline{\mathcal{G}}$ is amenable hence cannot contain 
free subgroup on two generators. It is also indicable. Hence it is a finitely presented indicable
group which is an ascending HNN extension with the finitely generated base group $\mathcal{G}$  that has finite
endomorphic presentation. Motivated by this, Grigorchuk in \cite{MR2195454} asked the following question:

\medskip
Is it correct that a finitely presented indicable group not containing
a free subgroup of rank 2 is an ascending HNN extension of a base group with finite
endomorphic presentation?\medskip

As a corollary of theorem 1, we provide an answer to this question under the stronger assumption that the group
has no free semigroup of rank 2: 
\bigskip

\textbf{Theorem 2} Let G be a  finitely presented indicable group not containing a free semigroup of rank 2. Then 
$G$ has the form of a semidirect product $H \rtimes \field{Z}$ where $H$ has ascending finite endomorphic 
presentation.

The reason why we need the stronger assumption is that in this case the kernel of the homomorphism
onto the infinite cyclic group itself is  finitely generated and hence theorem 1 can be applied.

\section{Definitions and Preliminaries}
Notation: 
\begin{itemize}
 \item If $G$ is a group and $X$ a subset then $\gp{X}$ denotes the subgroup 
of $G$ generated by $X$ and $\gp{X}^\#$ denotes the normal subgroup of $G$ generated
by $X$.
\item $X^\pm$ stands for the set $X \cup X^{-}$.
\item If $Y$ is a set of endomorphisms of a group, $Y^*$ stands for the free monoid generated by $Y$.
i.e. the closure of $\{1\}\cup Y$ under composition.
\item Unless stated otherwise, an equality means equality as words. We will indicate whenever necessary 
that some equality is thought to hold in some group.
\item If $w$ is an element of the free group on a set $X$ and $x \in X$, $exp_x(w)$ denotes the exponent sum
of $x$ in $w$.
\end{itemize}

We will frequently use the following fact also known as W.Dyck's theorem:

If $G$ is a group given as $F/N$ where $F$ is a free group and $N=\gp{R}^\#$ for some $R \subset F$,
then any map $$ \phi :F \longrightarrow H$$ to another group $H$ satisfying 
$\phi(r)=1$ in $H$ for all $r\in R$ induces a well defined group homomorphism
$$\phi: G \longrightarrow H$$

\begin{definition}
 An endomorphic presentation (or an L-presentation) is an expression

\begin{equation}
 \gp{X\mid Q\mid R \mid \Phi}
\end{equation}

where $X$ is a set, $Q,  R $ are subsets of the free group $F(X)$ on the set $X$ and $\Phi$ is a 
set of endomorphisms of $F(X)$. The expression (1) defines a group

$$G=F(X)/N $$ where 
$$N=\gp{Q \cup \bigcup_{\phi \in \Phi^*}\phi(R)}^\# $$

It is called a finite endomorphic presentation (or a finite L-presentation) if $X,Q,R,\Phi$ are all finite and  \textit{ascending} if $Q$ is empty. 
It is called  \textit{invariant} if the endomorphisms in $\Phi$ induce endomorphisms of $G$. Note that ascending L-presentations are invariant, but not all
finite L-presentations are invariant (see \cite{hartung}).
\end{definition}

(Some authors prefer to reserve the name L-presentation to the case where $\Phi$ only contains a single endomorphism.
 We will not make such a distinction and use both names).

Clearly all finite presentations are finite L-presentations. As mentioned in the introduction there are 
groups (such as the Grigorchuk group) which are not finitely presented but finitely L-presented. Also a
counting argument shows that most groups are not finitely L-presented. For general properties of L-presentations see \cite{MR2009317} and also the recent article 
\cite{hartung} where a variant of the Reidemeister-Schreier procedure is proven for finitely L-presented groups.

\medskip

We cite some auxiliary lemmas which we will use later:

\begin{lem} (See \cite{MR1277124})
 If a  group $G$ has no free subsemigroup of rank 2, then for all $a,b \in G$  
the subgroup

$$\gp{b^{-n}ab^n \mid n \in \field{Z}}$$

is finitely generated.

\end{lem}

\begin{lem} (See \cite{MR0387420})
Let $G$ be a finitely generated group and $H$ a normal subgroup such that $G/H$ is solvable.
If for all $a,b \in G$ the subgroup  $\gp{b^{-n}ab^n \mid n \in \field{Z}}$
is finitely generated, then $H$ is finitely generated.
\end{lem}

Lemma 1 and Lemma 2 together give:

\begin{lem}
Let $G$ be a finitely generated group not containing free subsemigroup of rank 2. If $G/H$ is solvable then $H$
is finitely generated. 
\end{lem}

\section{Proof of  Theorems}

\begin{thm}
Let $G$ be a finitely presented group. Let $H$ be a finitely generated normal
subgroup such that $G/H$ is infinite cyclic. Then $H$ has ascending finite L-presentation with two free group endomorphisms.
\end{thm}

\begin{proof}
Suppose that for $t\in G$ we have  $G/H=\gp{tH}$, then $G$ has the form of a semidirect product $G=H \rtimes \gp{t}$.

\medskip
From Neumann's Theorem \cite{MR1243634} (Page 52 ) it follows that $G$ has a presentation of the form $$G=\gp{t,a_1,\ldots,a_m \mid r_1,\ldots,r_n} $$

where $$H=\gp{a_1,\ldots,a_m}^\#_G$$ and $$exp_t(r_k)=0 $$

Consequently, the set $$T=\{ t^i \mid i \in \field{Z}\} $$ is a right Schreier transversal for $H$ in $G$.

\smallskip

Following the Reidemeister-Schreier process for $H$, we can take the elements 
$$a_{j,i}=t^{-i}a_jt^{i} \qquad  \quad  j=1,\ldots,m \;\; i \in \field{Z}$$ as generators for $H$  and the words

$$r_{k,i}=\rho(t^{-i}r_kt^{i})\qquad \quad k=1,\ldots,n \;\; i \in \field{Z} $$
as relators, where $\rho$ is the rewriting of $t^{-i}r_kt^{i}$ as a word in the $a_{j,i}$' s. So, $H$ has the presentation 

\begin{equation}
 H=\gp{a_{j,i} \quad (j=1,\ldots,m \;\; i \in \field{Z}) \mid  r_{k,i} \quad (k=1,\ldots,n \;\; i \in \field{Z})}
\end{equation}
Each $ r_k$ is a word of the form 
$$r_k=\prod_{s=1}^{n_k} t^{-l_s}a_{z_s}t^{l_s}$$
where $a_{z_s} \in \{a_j, j=1,\ldots,m \}^\pm $ and $n_k\in\field{N},\;\; l_s \in \field{Z} $. Therefore we have

$$r_{k,0}=\rho(r_k)=\rho(\prod_{s=1}^{n_k} t^{-l_s}a_{z_s}t^{l_s})
=\prod_{s=1}^{n_k} a_{z_s,l_s}$$
 and
\begin{equation}
 r_{k,i}=\rho(t^{-i}r_kt^i)=\prod_{s=1}^{n_k} a_{z_s,l_s+i}\quad i \in \field{Z}
\end{equation}

The map $$ s : H \longrightarrow H $$ defined by $s(h)=t^{-1}ht$ is clearly an automorphism of $H$. With respect to  presentation (2) 
of $H$, $s$ becomes $s(a_{j,i})=a_{j,i+1}$.

\medskip
Let $F$ be the free group on $\{a_{j,i} \quad j=1,\ldots,m \;\; i \in \field{Z} \}$. We will denote again by $s$ the automorphism of $F$ 
sending $a_{j,i}$ to $a_{j,i+1}$.

\medskip
Since by assumption $H$ is finitely generated, we can select a big enough natural number $N$ with the following properties:

\begin{itemize}
 \item $H=\gp{a_{j,i} \quad (j=1,\ldots,m)\quad |i|\leq N}$
\item Each word $r_{k,0}$ is a word in $\{a_{j,i} \quad j=1,\ldots,m\quad 
|i|\leq N \}^{\pm}$
\end{itemize}

So, each $a_{j,i}$  can be represented by a word in the finite generating set $\{a_{j,i} \quad j=1,\ldots,m\quad |i|\leq N \}^{\pm}$.

\medskip

For each $a_{j,i}$ we will recursively  construct a word $\gamma(a_{j,i})$ in this new finite generating set  which represents $a_{j,i}$ in $H$.

\medskip

For $a_{j,i}$ with $|i|\leq N$ we simply define $\gamma(a_{j,i})$ to be $a_{j,i}$.

\medskip

Pick $\gamma(a_{j,N+1})$ and $\gamma(a_{j,-(N+1)})$ two words in $\{a_{j,i}\mid  j=1,\ldots,m\quad |i|\leq N \}^{\pm}$
representing $a_{j,N+1}$ and $a_{j,-(N+1)}$ in $H$ respectively.

\medskip

For $i \geq N+1$ we define $\gamma(a_{j,i+1})$ recursively as follows:

$$\gamma(a_{j,i+1})=\gamma(s(\gamma(a_{j,i})))$$

(for a word $w$, we define $\gamma(w)$ as the word obtained by applying $\gamma$ to each letter of $w$). Note that $s(\gamma(a_{j,i}))$ 
is a word in $\{a_{j,i}\mid  j=1,\ldots,m\quad |i|\leq N+1 \}^{\pm}$ therefore we can apply $\gamma$ to it.

\medskip

Similarly for $i \leq -(N+1)$ we define $\gamma(a_{j,i-1})$ as

$$\gamma(a_{j,i-1})=\gamma(s^{-1}(\gamma(a_{j,i}))) $$

Defining $\gamma$ as above gives the following equalities in the free group $F$:

\begin{equation}
\gamma(a_{j,i+1})=\gamma(s(\gamma(a_{j,i}))) \quad \text{for} \quad i \geq -N
\end{equation}

and
\begin{equation}
  \gamma(a_{j,i-1})=\gamma(s^{-1}(\gamma(a_{j,i}))) \quad \text{for} \quad i \leq N
\end{equation}

\begin{lem}
 $H$ has the presentation 
$$\gp{a_{j,i} (j=1,\ldots ,m\quad |i|\leq N) \mid \gamma(r_{k,i}) (k=1,\ldots,n\quad i \in \field{Z})} $$
\end{lem}

\begin{proof}
 This follows by Tietze transformations, but we will explicitly construct an isomorphism between these presentations. In order to avoid confusion, 
we denote elements in the asserted presentation with bars and set
$$\overline{H}= \gp{\overline{a_{j,i}} (j=1,\ldots ,m\quad |i|\leq N) \mid \overline{\gamma(r_{k,i})} (k=1,\ldots,n\quad i \in \field{Z})} $$
We will show that $\overline{H}\cong H$ using the presentation (2) of $H$. For this define: 
$$
\begin{array}{cccc}
 \varphi:&  H &  \longrightarrow & \overline{H} \\
 & a_{j,i}& \mapsto & \overline{\gamma(a_{j,i})}
\end{array}
$$

We have $\varphi(r_{k,i})=\overline{\gamma(r_{k,i})}=1$ in $\overline{H}$. So $\varphi$ maps relators of $H$
to relators in $\overline{H}$ and  hence is a well  defined group homomorphism. Conversely  define :

$$
\begin{array}{cccc}
 \psi:&  \overline{H} &  \longrightarrow & H\\
 & \overline{a_{j,i}}& \mapsto & a_{j,i}
\end{array}
$$

Since $\gamma(a_{j,i})=a_{j,i}$ in $H$ we have

  $$\psi(\overline{\gamma(r_{k,i})})=\gamma(r_{k,i})=r_{k,i}=1\quad \text{in}\quad H$$  
which shows that $\psi$ is a well defined group homomorphism.

Finally the following equalities show that $\varphi$ and $\psi$ are mutual inverses:
$$
(\varphi \circ \psi)(\overline{a_{j,i}})=\varphi(a_{j,i})=\overline{\gamma(a_{j,i})}=\overline{a_{j,i}}
$$
 
(where the last equality is true since $|i|\leq N$ in this case.)

$$(\psi \circ \varphi)(a_{j,i})=\psi(\overline{\gamma(a_{j,i})})=\gamma(a_{j,i})=a_{j,i} \quad 
\text{in}\quad H $$

Hence  $\overline{H}$ is isomorphic to $H$.

$\Box$
\end{proof}

 Let $F_r$ be the free group with generators $\{a_{j,i}\mid  j=1,\ldots,m\quad |i|\leq N \}$.
Define two endomorphisms $\eta$ and $\tau$ of $F_r$  as follows:
$$ \eta(a_{j,i})=\gamma(s(a_{j,i}))=\gamma(a_{j,i+1})$$ and

$$ \tau(a_{j,i})=\gamma(s^{-1}(a_{j,i}))=\gamma(a_{j,i-1})$$

where $\gamma$ is as above. Note that $\eta$ and $\tau$ induce the automorphisms $s$ and $s^{-1}$ of $H$ respectively.

\begin{lem}
 
In $F_r$ we have the equality

$$
 \gamma(r_{k,i}) = \left\{
     \begin{array}{lr}
       \eta^i(r_{k,0}) & \text{ if} \quad i \geq 0\\
       \tau^{-i}(r_{k,0}) & \text{if} \quad i < 0
     \end{array}
   \right.
$$

\end{lem}

\begin{proof}
 Suppose $i\geq 0$. We use induction on $i$.

\medskip

If $i=0$, $\gamma(r_{k,0})=r_{k,0}$ by choice of $\gamma$ and the natural number $N$. 
Suppose the equality holds for $i$. Then

$$
\begin{array}{cclc}
\eta^{i+1}(r_{k,0})& = &\eta(\eta^i(r_{k,0}))& \\
\\
 &=&\eta(\gamma(r_{k,i}))&(\text{by induction hypothesis}) \\
\\
&= &\eta(\gamma(\prod a_{z_s,l_s+i}))&  (\text{using equation (3)})\\
\\
&=& \prod \eta(\gamma(a_{z_s,l_s+i}))&\\
\\
&=& \prod \gamma s \gamma (a_{z_s,l_s+i})&\\
\\
&=& \prod \gamma(a_{z_s,l_s+i+1})& (\text{using equation (4), since $|l_s|\leq N$} )\\
\\
&=& \gamma (\prod a_{z_s,l_s+i+1})&\\
\\
& =& \gamma(r_{k,i+1})
\end{array}
$$

A similar argument with induction on $-i$ (and using equation (5)) shows the required 
identity for $i<0$.
\end{proof}
$\Box$
\begin{lem}
 $H$ has the following ascending finite L-presentation:

$$\gp{a_{j,i}\quad (j=1,\ldots ,m\quad |i|\leq N) \mid r_{k,0}\quad  k=1,\ldots,n \mid \{ \eta ,\tau\}} $$
 
\end{lem}

\begin{proof}
 Again not to cause confusion we denote the asserted presentation with bars and set
$$\overline{H}=\gp{\overline{a_{j,i}}\quad (j=1,\ldots ,m\quad |i|\leq N) \mid \overline{r_{k,0}}\quad  k=1,\ldots,n 
\mid \{ \overline{\eta} ,\overline{\tau}\}} $$
where $\overline{\eta},\overline{\tau}$ are endomorphisms of the free group $\overline{F_r}$ 
analogous to $\eta$ and $\tau$. More precisely:

$$\overline{\eta}(\overline{a_{j,i}})=\overline{\eta(a_{j,i})} $$
$$\overline{\tau}(\overline{a_{j,i}})=\overline{\tau(a_{j,i})} $$

We will show that $\overline{H}\cong H$ and we will use the presentation of $H$

$$\gp{a_{j,i} (j=1,\ldots ,m\quad |i|\leq N) \mid \gamma(r_{k,i}) (k=1,\ldots,n\quad i \in \field{Z})} $$

which was found in  Lemma 4. To this end define:

$$
\begin{array}{cccc}
 \phi:&  H &  \longrightarrow & \overline{H} \\
 & a_{j,i}& \mapsto & \overline{a_{j,i}}
\end{array}
$$
 We have
$$
\phi(\gamma(r_{k,i})) = \overline{\gamma(r_{k,i})}=
\left\{
     \begin{array}{lr}
       \overline{\eta}^i(\overline{r_{k,0}}) & \text{ if} \quad i \geq 0\\
       \overline{\tau}^{-i}(\overline{r_{k,0}}) & \text{if} \quad i < 0
     \end{array}
   \right.
$$

by lemma 5. Hence $\phi$ is a well defined group homomorphism. Conversely define:

$$
\begin{array}{cccc}
 \chi:&  \overline{H} &  \longrightarrow & H\\
 & \overline{a_{j,i}}& \mapsto & a_{j,i}
\end{array}
$$
To show that $\chi$ is well defined, we need to prove that for all $\overline{f}\in \{\overline{\eta},
\overline{\tau}\}^*$
 and for all $k=1,\ldots,n$ we have:
$$ \chi(\overline{f}(\overline{r_{k,0}}))=1\quad \text{in}\quad H$$

This is true since $\eta$ and $\tau$ (and hence $f$) induce isomorphisms on $H$. This shows that 
$\chi$ is a well defined group homomorphism. Clearly $\phi$ and $\chi$ are mutual inverses.
\end{proof} $\Box$

Hence we have proven  theorem 1.
\end{proof}

\begin{thm}
 Let G be a  finitely presented indicable group not containing a free semigroup of rank 2. Then 
$G$ has the form of a semi direct product $H \rtimes \field{Z}$ where $H$ has ascending  finite  L-presentation.
\end{thm}

\begin{proof}
 Follows directly from theorem 1 and lemma 3.
\end{proof}
$\Box$

\bigskip

\textbf{Some Remarks:}

\smallskip

\textbf{1)} As mentioned in the introduction, groups with invariant finite L-presentations embed nicely 
into finitely presented groups via HNN extensions.
In our special case (i.e. a presentation for $H$ is obtained via theorem 1), the endomorphisms of 
the L-presentation of $H$ actually induce automorphism of $H$ and $H$ embeds into  $G$ as a normal subgroup. 

\medskip

\textbf{2)} Though all finitely generated recursively presented groups embed into finitely presented groups, I have been told by Mark Sapir (private communication) that not all finitely 
generated recursively presented groups embed 
into finitely presented groups \textit{as normal subgroups}. His example was the first Grigorchuk group. This shows that even finitely
L-presented groups may fail to be normal subgroups of finitely presented groups. This indicates that such groups have a rather restricted structure.
Hence a natural question is what additional 
structure finitely generated \textit{normal} subgroups of finitely presented groups have. One answer could be given if one can generalize theorem 1 
to arbitrary finitely generated normal subgroups. One would obtain a characterization in the following sense:

\smallskip

\textit{A finitely generated group is a normal subgroup of a finitely presented group if and only if it has an ascending finite L-presentation where the endomorphisms induce automorphisms of the group.}

\medskip

Therefore we would like to formulate the question  whether Theorem 1 can be generalized to arbitrary finitely generated normal subgroups.
 
\medskip
\textbf{3)} We would like to present a concrete example in which Theorem 1 can be used. This is also a counter example to the assertion
(as written in \cite{MR2009317} Theorem 2.16) that all finitely L-presented groups have the Schur Multiplier the direct product of finitely generated
abelian groups. Upon discussing with the author of \cite{MR2009317} it was observed that one needs one additional hypothesis.

\medskip

Let $G$ be the group given by the presentation $$\gp{a,b,t,u \mid [a,b],[a,u],[t,b],[t,u],a^t=a^2,b^u=b^2}$$ which is the direct square
of the Baumslag-Solitar group $BS(1,2)$. Let $z=tu^{-1}$ and consider the subgroup $H=\gp{a,b,z}$ which is normal and has infinite
cyclic quotient. Then following theorem 1 one arrives at the following finite L-presentation for $H$:

$$\gp{a,b,z  \mid[a,b],a^z=a^2, (b^2)^z=b\mid \{\eta,\tau\}} $$

where 

$$
\eta = \left\{
\begin{array}{ccc}
   a & \mapsto & a^2\\
   b & \mapsto  & b \\
   z & \mapsto & z \\
  \end{array}
\right.$$

and

$$
\tau = \left\{
\begin{array}{ccc}
   a & \mapsto & zaz^{-1}\\
   b & \mapsto  & b \\
   z & \mapsto & z \\
  \end{array}
\right.$$

\medskip

Now since $BS(1,2)=\field{Z}[\frac12] \rtimes \field{Z}$ we have $H=\field{Z}[\frac12]^2 \rtimes \field{Z}$ and using Shapiro's
lemma one can see that $H_2(H,\field{Z}) \cong \field{Z}[\frac 12]$.

\medskip

\textbf{4)} Another problem of interest is the structure of finitely generated subgroups of finitely L-presented groups. For finite index subgroups  one 
has a Reidemeister-Schreier algorithm to compute a finite L-presentation for the subgroup (See \cite{hartung}).
For other subgroups it would be nice to investigate whether analogous  statements similar to Theorem 1 hold.
\medskip

\textbf{Acknowledgements:}

\medskip

I wish like to thank my advisor Rostislav Grigorchuk for his valuable comments and helpful discussions. 
I also want to thank Laurent Bartholdi for nice remarks and suggestions. I am grateful to Ben Wieland for suggesting
the example in remark 3.

\bibliography{GMJ-11-0168.bib}
\end{document}